\documentclass [twoside,10pt]{article}
\usepackage{graphicx}
\usepackage{amssymb}
\usepackage{amsthm}
\newtheorem{thm}{Theorem}

\newcommand{\bd}{{\rm bd}}
\newcommand{\diam}{{\rm diam}}
\newcommand{\conv}{{ \rm conv}}
\newcommand{\width}{{\rm width}}
\newcommand{\dist}{{\rm dist}}

\newcommand{\sh}{{\rm sinh \,}}
\newcommand{\ch}{{\rm cosh \,}}
\newcommand{\arsh}{{\rm arcsinh \,}}

\title{\bf Applications of equidistant supporting surfaces of a convex body in the hyperbolic space}

\date{}

\begin{document}

\maketitle

\thispagestyle{empty}

\vskip-1cm

\centerline
{\author{MAREK LASSAK}}

\pagestyle{myheadings} \markboth{\centerline {Marek Lassak}}{\centerline {Applications of equidistant supporting surfaces in the hyperbolic space}}

\baselineskip 17pt 

\maketitle
\vskip 0.6cm

\noindent 
{\bf Abstract}.
For a hyperplane $H$ supporting a convex body $C$ in the hyperbolic space $\mathbb{H}^d$ we define the width of $C$ determined by $H$ as the distance between $H$ and a most distant ultraparallel hyperplane supporting $C$. 
The thickness (i.e., the minimum width) of $C$ is denoted by $\Delta(C)$. 
A convex body $R \subset \mathbb{H}^d$ is called reduced if for every body $Z \subsetneq R$ we have $\Delta(Z) < \Delta(R)$.
We show that for any extreme point $e$ of a reduced body $R \subset \mathbb{H}^d$ there exists a supporting hyperplane $H$ of $R$ which passes through $e$ or its equidistant surface supporting $R$ passes through $e$.
Bodies of constant width in $\mathbb{H}^d$ are defined as bodies all whose widths are equal.
We prove that every complete body in $\mathbb{H}^d$ is a body of constant width.

\baselineskip 19pt 

\vskip0.2cm
\noindent 
{\bf Mathematical Subject Classification (2010).} 52A55. 

\vskip0.2cm
\noindent
{\bf Keywords.} Hyperbolic geometry, equidistant surface, convex body, width, constant width, thickness, reduced body, 
complete body. 

\vskip0.5cm

\date{}

\maketitle

\section{Introduction}

Let $H$ be a hyperplane supporting a convex body $C$ in the hyperbolic space $\mathbb{H}^d$. 
We define the {\bf width of $C$ determined by $H$} as the distance between $H$ and any farthest ultraparallel hyperplane supporting $C$ (see \cite{[L23]}).
Since $C$ is compact, there exists at least one such a most distant hyperplane (sometimes there are a finitely or even infinitely many of them). 
The symbol ${\rm width}_H (C)$ denotes this width of $C$ determined by $H$.

By the {\bf thickness $\Delta (C)$ of} a convex body $C \subset \mathbb{H}^d$ we mean the infimum of ${\rm width}_H (C)$ over all hyperplanes $H$ supporting $C$.
By compactness arguments, this infimum is realized, so $\Delta (C)$ is the minimum of the numbers ${\rm width}_H (C)$.

Recall that an {\bf equidistant surface to} $H$ is the set of all points in a fixed positive distance which are in one half-space bounded by $H$.
By the {\bf equidistant strip} generated by a hyperplane $H$ and an equidistant surface $E$ to it we mean the set $\conv (H \cup E)$, where the symbol ``$\conv$" means creating the convex hull.
The {\bf thickness} of this equidistant strip is defined as the distance between $H$ and $E$.

Clearly, for a convex body $C \subset \mathbb{H}^d$ and a supporting hyperplane $H$ of $C$ there exists a unique nearest equidistant surface $E_H(C)$ to $H$ such that $C$ is a subset of the equidistant strip being the convex hull of $H \cup E_H(C)$. 
We say that $E_H(C)$ {\bf supports} $C$.

A convex body $R \subset \mathbb{H}^d$ is called {\bf reduced} if for every convex body $Z$ properly contained in $R$
we have $\Delta(Z) < \Delta(R)$. 
This notion for $\mathbb{H}^d$ is introduced in \cite{[L23]} in analogy to the notion of a reduced body in the Euclidean $d$-dimensional space $\mathbb{E}^d$ introduced by Heil in \cite{[He]}.
Examples of reduced bodies in $\mathbb{H}^d$ are bodies of constant width, as shown in Proposition 4 of \cite{[L23]}.
Some reduced polygons in $\mathbb{H}^2$ are presented in \cite{[L24]} (in particular, the regular odd-gons are reduced).

The first aim of this note is to prove that for any reduced body $R \subset \mathbb{H}^d$ and an arbitrary extreme point $e \in R$ there exists a supporting hyperplane $H$ of $R$ such that it supports $R$ at $e$ or its equidistant surface supports $R$ at $e$.

If for every hyperplane supporting a convex body $W \subset \mathbb{H}^d$ the width of $W$ determined by this hyperplane is $\delta$, we say that $W$ is a {\bf body of constant width} $\delta$. 

Similarly to the traditional notion of a complete set in $\mathbb{E}^d$ (for instance, see the books \cite{[BF]} by Bonnesen and Fenchel, \cite{[CG]} by Chakerian and Groemer and \cite{[Eg]} by Eggleston) we say that a set $C \in \mathbb{H}^d$ of diameter $\delta$ is {\bf complete} provided $\diam (C \cup \{x\}) > \delta$ for every $x \not \in C.$

The second aim is to prove that every complete body of diameter $\delta$ in the hyperbolic space $\mathbb{H}^d$ is a body of constant width~$\delta$.
The proof of this result also uses the concept of supporting equidistant surface.

\section{Supporting at extreme points of a reduced body}

The following theorem is an analog of Theorem 1 of \cite{[L90]} that through every extreme point of a reduced body $R$ in $\mathbb{E}^d$ a supporting hyperplane $H$ passes such that $\width_H(R) = \Delta(R)$.
This analogy is not literal since for some extreme points $e$ of some reduced bodies $R \subset \mathbb{H}^d$ no supporting hyperplane $H$ with $\Delta_H(R) = \Delta(R)$ passes through $e$.
An example of such an $R$ is given after the following theorem for $e^+$ as the extreme point (the author is not able to find another such example).

Let us recall the following fact which is applied in the below proof.

Proposition 2 of \cite{[L23]}.
{\it Let $C \subset \mathbb{H}^d$ be a convex body and $H$ any supporting hyperplane of $C$.
Then ${\rm width}_H (C)$ equals to the distance between $H$ and the nearest equidistant surface $E$ to $H$ such that $C$ is a subset of the equidistant strip being the convex hull of $H \cup E$.}

\begin{thm} \label{extreme point}
Let $e$ be an extreme point of a reduced body $R \subset \mathbb{H}^d$.
Then there exists a supporting hyperplane $H$ of $R$ such that $H$ or $E_H(R)$ supports $R$ at $e$.
The thickness of the equidistant strip $\conv (H \cup E_H(R))$ is $\Delta(R)$.
\end{thm}

\begin{proof}
Take the open ball $B_i(e)$ of radius $\Delta(R)/i$ centered at $e$ for $i=2,3, \dots$.
Then $R_i = \conv (R \setminus B_i)$ is a convex body properly contained in $R$.
Since $R$ is reduced, we have $\Delta (R_i) < \Delta (R)$ for $i= 2,3, \dots$.
Applying the just recalled Proposition 2 of \cite{[L23]} and having in mind our definition of the thickness of the equidistant strip we conclude that there exists a supporting hyperplane $H_i$ of $R_i$ and its $E_{H_i}(R_i)$ such that the thickness of the equidistant strip $\conv (H_i \cup E_{H_i}(R_i))$ equals $\Delta (R_i)$.
For every $i \in \{2,3,\dots\}$ there are two possibilities:

(1) when $H_i$ supports $R_i$ at a boundary point of it in $B_i(e)$,

(2) when $H_i$ does not supports $R_i$ at a boundary point of it in $B_i(e)$.

There exist infinitely many $H_i$ in the sequence $H_2, H_3, \dots$ which fulfill (1) or infinitely many $H_i$ in this sequence which fulfill (2).

Case 1. By compactness arguments, if there are infinitely many $H_i$ which fulfill (1), then there is a subsequence of the sequence $H_2, H_3 \dots$ whose limit is a hyperplane $H$ supporting $R$ at $e$.
Since the thickness of $\conv (H_i \cup E_{H_i}(R_i))$ is $\Delta(R_i)$, we conclude that the thickness of the equidistant strip $\conv (H \cup E_H(R))$ is $\Delta (R)$.

Case 2. Analogously, when there are infinitely many $H_i$ which fulfill (2), there is a subsequence of this sequence whose limit is a hyperplane $H$ supporting $R$ and the limit of $E_{H_i}$ is $E_H(R)$.
Since $\width_{H_i}(R_i) = \Delta(R_i) < \Delta (R)$, we see that every $E_{H_i}(R_i))$ from our subsequence passes through $B_i$. 
Since for every $i$ from our subsequence the equidistant surface $E_{H_i}$ intersects $B_i$, we conclude that $E_H(R)$ passes through $e$.
We see that again the thickness of the strip $\conv (H \cup {H_i}(R_i))$ is $\Delta (R)$.

From the two cases we conclude the thesis of our theorem.
\end{proof}

Let us show examples of the cases from this proof.

The author learned about a reduced rhombus from K. Jr. B\"or\"oczky, A. Freyer and \'A. Sagmeister, who are preparing a paper on reduced bodies in $\mathbb{H}^2$. 
We mean the convex hull $P$ of two perpendicular segments intersecting each other at the common midpoint $m$ such that the longer one is sufficiently long with respect to the shorter one. 
Denote by $e^-$ an end-point of the shorter segment $S^-$ and by $e^+$ an end-point of the longer segment $S^+$.
Consider the extreme point $e^-$ of $P$.
Every $H_i$ passes through $B_i(e^-)$ and $E_{H_i}(P)$ passes through $e^+$.
For the straight line $H$ through $e^-$ orthogonal to $S^-$ we have $\width_H(P) = \dist(e^+,H) = \Delta(P)$, which means that we have Case 1 (but not Case 2).
Take into account the extreme point $e+$ of $P$.
Now every $H_i$ passes through $e^-$ and $E_{H_i}(P)$ passes through $B_i(e^+)$.
We have $\width_H(P) = \dist(e^+,H) = \Delta(P)$.
Thus $H$ supports $P$ at $e^-$ and $E_H(P)$ passes through $e^+$, which means that we have Case 2 (but not Case 1).
An analogous situation in $\mathbb{H}^d$ is for a sufficiently long crosspolytope, i.e., the convex hull of $d$ perpendicular segments intersecting each other at midpoints.
Again we take $e^-$ as an end-point of a shortest segment and $e^+$ as an end-point of a longest segment.
By the way, the reader can easily check that if the shorter one has length $\lambda$, then the longer should be of length over $\arsh(\sh \lambda \cdot \ch 2\lambda)$ in order the rhombus, and so the crosspolytope, be reduced.
Hint: apply the Lambert quadrilateral $e^-me^+p$, where $p$ is the projection of $e^+$ on $H$.
Let us add that we have both cases simultaneously for any body of constant width in $\mathbb{H}^d$ and any of its boundary point as $e$.

\section{Complete bodies are of constant width}

The content of this section appeared in version 3 of the arXiv preprint of [12].
In version 4 it was removed, in order to make that version the same as the one in journal.

The following fact is applied in the proofs of the forthcoming Claim and Lemma, both needed for the proof of Theorem 2.

Claim 4 from \cite{[L23]}. 
{\it  Every complete body $C$ of a diameter $\delta$ in $\mathbb{H}^d$ coincides with the intersection of all balls of radius $\delta$ centered at points of $C$.}

\vskip0.25cm
\noindent
{\bf Claim.}
{\it Let $C \subset \mathbb{H}^d$ be a complete body and let $H$ be a supporting hyperplane of $C$.
Then the equidistant surface $E_H(C)$ contains exactly one point of $C$.}

\begin{proof}
Imagine the opposite.
Then at least two points $r_1$ and $r_2$ of $E_H(C)$ belong to $C$.
By the just recalled Claim 4 of \cite{[L23]} our complete body $C$, whose diameter denote by $\delta$, is a subset of the intersection of all balls of radius $\delta$ centered at points of $C$.
Hence $C$ is a subset of the intersection of the two balls of radius $\delta$ centered at $r_1$ and $r_2$.
Observe that none of points of the intersection of these two balls belongs to $H$.
So $C \cap H = \emptyset$.
A contradiction with the assumption that $H$ supports $C$.
This ends the proof.
\end{proof}

For different points $a, b \in \mathbb{H}^d$ at a distance $\delta$ from a point $c \in \mathbb{H}^d$ define the piece $P_c(a,b)$ of circle as the set of points $g \in \mathbb{H}^d$ such that the segment $cg$ has length $\delta$ and intersects $ab$. 
This notion is similar to a definition from the paper \cite{[Je]} by Jessen and also a definition in the part ${\rm (ii)} \Longrightarrow {\rm (i)}$ of the proof of Theorem 52 of the book \cite{[Eg]} by Eggleston for $\mathbb{E}^d$.

\vskip0.1cm
The proofs of our below Lemma and Theorem 2 are analogous to the proofs for $\mathbb{E}^d$ by Meissner \cite{[Me]} for $d=2,3$ and Jessen \cite{[Je]} for arbitrary $d$.
They are also analogous to the spherical version given in \cite{[L20]} and \cite{[L22]}.

\vskip0.3cm
\noindent
{\bf Lemma.}
{\it Let $C \subset \mathbb{H}^d$ be a complete body of diameter $\delta$. 
Consider $P_c(a,b)$ with $|ac| = |bc| = \delta$, where $a, b \in C$ and $c \in \mathbb{H}^d$.
Then $P_c(a,b) \subset C$.}

\begin{proof}
We start with confirming the thesis for a ball $B$ of radius $\delta$ in place of $C$.
Unique $\mathbb{H}^2 \subset \mathbb{H}^d$ exists such that $a, b, c \in \mathbb{H}^2$.
Let $D= B \cap \mathbb{H}^2$.
Take the circle containing $P_c(a,b)$ and the disk $D'$ bounded by it.
Denote by $a', b'$ the points of the intersection of this circle with the circle bounding $D$.
Clearly, $P_c(a,b) \subset P_c(a',b')$.
Observe that the radius of $D'$ is also $\delta$ and that the distance of the centers of $D$ and $D'$ is at most $\delta$. 
Hence $P_c(a',b') \subset D$. 
Thus $P_c(a,b) \subset P_c(a',b') \subset D \subset B$.

By the preceding paragraph and the recalled earlier Claim 4 from \cite{[L23]} we obtain the thesis of the present lemma.
\end{proof}

Lemma is proved for arbitrary $d$ despite we will apply only its case for $d=2$.

In the proof of our Theorem 2 we apply the following three facts, which are below recalled for the convenience of the reader.
 
Theorem 1 of \cite{[L23]}.
{\it For every convex body $C \subset \mathbb{H}^d$ the maximum of ${\rm width}_H (C)$ over all hyperplanes $H$ supporting $C$ equals to the diameter of $C$.}

Theorem 2 of \cite{[L23]}. {\it Let $C \subset \mathbb{H}^d$ be a convex body and let $H$ be a supporting hyperplane of $C$ such that $\width_H (C) = \Delta (C)$.
Assume that there exists a unique most distant point $j \in C$ from $H$.
Then the projection of $j$ onto $H$ belongs to $H \cap C$.}

Proposition 5 of \cite{[L23]}.
{\it If $C \subset H^d$ is a complete body of diameter $\delta$, then for every $p \in \bd (C)$ there exists $p' \in C$ such that $|pp'|=\delta$.}

\begin{thm}
Every complete body of diameter $\delta$ in $\mathbb{H}^d$ is a body of constant width $\delta$.
\end{thm}

\begin{proof}
Suppose that our thesis does not hold true, i.e., that $\width_H(C) \not = \delta$ for a hyperplane $H$ supporting $C$. 
By Theorem 1 of \cite{[L23]} recalled above, 
we cannot have $\width_H(C) > \delta$.
Consequently, $\width_H(C) < \delta$ and thus, having in mind the definition of the thickness, we obtain $\Delta (C) < \delta$. 
Since $C$ is complete, by our Claim the equidistant surface $E_H(C)$ contains exactly one point of $C$.
Denote it by $j$.
By Theorem 2 of \cite{[L23]} (recalled above) the projection $h$ of $j$ onto $H$ belongs to $C$.
By Proposition 5 of \cite{[L23]} (recalled above) there exists $j' \in C$ at distance $\delta$ from $j$. 
Since the triangle $jhj'$ is non-degenerate (still $|hj| = \width_H(C) < |j'j|$), there is a unique two-dimensional hyperbolic plane $\mathbb{H}^2 \subset \mathbb{H}^d$ containing it.
Clearly, $jhj'$ is a subset of the convex body $M= C \cap \mathbb{H}^2$ of $\mathbb{H}^2$.
Denote by $F$ this half-plane of $\mathbb{H}^2$ for which $hj' \subset \bd(F)$ and $j \in F$.
There exists a unique $c \in F$ such that $|ch| = \delta = |cj'|$.
Since $c$ belongs to the intersection of the two circles of radius $\delta$ centered at $h$ and $j'$, we conclude that $c$ belongs to this half-plane $G$ of $\mathbb{H}^2$ bounded by the straight line $K$ through $h$ and $j$ which contains $j'$.
By Lemma for $d=2$ we get $P_c(h,j') \subset M$.

By $|ch| = \delta$, $|cj'| = \delta$, $|jj'| = \delta$, $c \in F$ and $c \not = j$ we conclude that $c \not \in K$.
Consequently, $c$ belongs to the interior of $G$.
Hence $P_c(h,j')$ intersects the line $\bd(M \cap \mathbb{H}^2)$ at a point $h'$ different from $h$ and $j'$.
So the set $P_c(h,j') \setminus \{h,h'\}$ does not intersect $M$. 
This contradicts $P_c(h,j') \subset M$ established earlier.
Consequently, $C$ is a body of constant width~$\delta$.
\end{proof}

From the above Theorem we see that in Theorem 3 of \cite{[L23]} also the implication $(c) \Longrightarrow (a)$ holds true.
As a consequence, conditions $(a), (b)$ and $(c)$ in the mentioned Theorem 3 of \cite{[L23]} are equivalent.
In other words, complete bodies, bodies of constant width and bodies of constant diameter in $\mathbb{H}^d$ coincide.
This is analogous to the classical fact in Euclidean  and also in spherical space (see \cite{[L20]} and  \cite{[L22]}).

Analogs of Theorem 2 are also proved for some other notions of constant width in $\mathbb{H}^d$, for instance see the papers  \cite{[BS]} by B\"or\"oczky and Sagemeister, and \cite{[De]} by Dekster.

\vskip0.2cm
\noindent
Marek Lassak,

\noindent
University of Science and Technology,

\noindent
85-789 Bydgoszcz, Poland,

\noindent
e-mail: lassak@pbs.edu.pl

\end{document}